\documentclass[11pt]{amsart}
\usepackage{bbm}
\usepackage[all]{xy}

\newtheorem{theorem}{Theorem}[section]

\theoremstyle{definition}

\newtheorem{conj}[theorem]{Conjecture}

\theoremstyle{remark}

\numberwithin{equation}{section}

\def\im{\mathop\mathrm{Im}\nolimits}

\def\Hom{\mathop\mathrm{Hom}\nolimits}
\def\Ext{\mathop\mathrm{Ext}\nolimits}
\def\CM{\mathop\mathsf{CM}\nolimits}

\def\al{\alpha}			\def\be{\beta}
\def\la{\lambda}		\def\de{\delta}
\def\eps{\varepsilon}

\def\aK{\mathbbm{k}}	\def\mN{\mathbb{N}}
\def\fW{\mathbf{w}}		\def\fE{\mathbf{e}}
\def\fU{\mathbf{u}}

\def\qR{{\boldsymbol R}}		\def\qS{{\boldsymbol S}}
\def\qI{{\boldsymbol I}}		\def\qE{{\boldsymbol E}}
\def\qK{{\boldsymbol K}}

\def\rX{\mathrm{X}}

\def\cE{\mathcal E}
\def\kE{\mathcal{E}} 	\def\kF{\mathcal{F}}

\def\bop{\bigoplus}

\def\+{\oplus}
\def\*{\otimes}
\def\8{\infty}
\def\mt{\mbox{-}}

\def\sbe{\subseteq}

\def\xx{\times}

\def\mtr#1{\begin{pmatrix}#1\end{pmatrix}}

 \newcommand{\hhh}{\rule{0pt}{9pt}}

\def\cm{Cohen--Macaulay}
\def\iff{if and only if }

\begin{document}

\title[Cohen-Macaulay modules over $T_{44}$]{On Cohen-Macaulay modules over the plane curve singularity of type $T_{44}$}

\author[Y. Drozd]{Yuriy A. Drozd}
\address{Institute of Mathematics, National Academy of Sciences, 01601 Kyiv, Ukraine}
\email{y.a.drozd@gmail.com,\,drozd@imath.kiev.ua}
\urladdr{http://www.imath.kiev.ua/$\sim$drozd}

\author[O. Tovpyha]{Oleksii Tovpyha}
\address{Institute of Mathematics, National Academy of Sciences, 01601 Kyiv, Ukraine}
\email{tovpyha@gmail.com}


\begin{abstract}
For a wide class of Cohen--Macaulay modules over the local ring of the plane curve singularity of type $T_{44}$ we explicitly describe the 
corresponding matrix factorizations. The calculations are based on the technique of matrix problems, in particular, representations of bunches 
of chains.
\end{abstract}
\maketitle

\section{Introduction}
\label{s1} 

 Let $\aK$ be an algebraically closed field, $\qS=\aK[[X,Y]]$. Recall that the complete local ring of a plane curve singularity of type $T_{44}$
 is $\qR=\qS/(F)$, where $F=XY(X-Y)(X-\la Y)$ and $\la\in\aK\setminus\{0,1\}$. In \cite{die} a classification of maximal \cm\ modules over the 
 ring $\qR$ was 
 obtained in terms of the Auslander--Reiten quiver. Another way to get such a classification was proposed in \cite{dg} through the technique
 of \emph{matrix problems} and in \cite{bikr} using cluster-tilting. Nevertheless, in these papers there was no explicit description of these 
 modules in terms of generators and relations, or, the same, \emph{matrix factorizations} of the polynomial $F$ \cite{eis,yo}. In this paper we 
 present such a description for a wide class of $\qR$-modules. 
 
 We consider $\qR$ as the subring of the direct product $\qR_1\xx\qR_2\xx\qR_3\xx\qR_4$, where all $\qR_i=\aK[[t]]$, generated by the 
 elements $x=(t,0,t,\la t)$ and $y=(0,t,t,t)$.  
 We denote by $\qR_{ij}$ the projection of $\qR$ to $\qR_i\xx\qR_j$. All rings $\qR_{ij}$ are isomorphic to $\aK[[X,Y]]/(XY)$, hence all
 indecomposable $\qR_{ij}$-modules are $\qR_i,\,\qR_j$ and $\qR_{ij}$. 
 Let $\qK_i\simeq\aK((t))$ be the field of fractions of $\qR_i$, $\qK_{ij}=\qK_i\xx\qK_j$. Every \cm\ $\qR$-module $M$ embeds into 
 $\qK\*_\qR M$. Denote by $N$ the image of $M$ under the projection $\qK\*_\qR M\to\qK_{12}\*_\qR M$  
 and by $L$ be the kernel of the surjection $M\to N$. Then $N\in\CM(\qR_{12})$ and $L\in\CM(\qR_{34})$.
  The following result, though elementary, is the background of this paper.

 \begin{theorem}\label{bimod} 
  There is an equivalence of the category $\CM(\qR)$ with the category $\cE$ of elements of the $\CM(\qR_{12})\mt\CM(\qR_{34})$-bimodule
  $\Ext^1_\qR$ \emph{(in the sense of \cite{sur})}.
 \end{theorem}
 \begin{proof}
  We have seen that for every $M\in\CM(\qR)$ there is an exact sequence $0\to L\to M\to N\to 0$, where 
  $N\in\CM(\qR_{12}),\,L\in\CM(\qR_{34})$. It defines an element $\xi\in\Ext^1_\qR(N,L)$. Let $0\to L'\to M'\to N'$ be another
  exact sequence with $N'\in\CM(\qR_{12}),\,L'\in\CM(\qR_{34})$, $\xi'\in\Ext^1_\qR(N',L')$ be the corresponding element and
  $f\in\Hom_\qR(M,M')$. Then $f(L)\sbe L'$, so we obtain a commutative diagram
   \[
   \xymatrix{0\ar[r] & L \ar[r]\ar[d]^\be  & M \ar[r]\ar[d]^f  & N \ar[r]\ar[d]^\al &0 \\
    0\ar[r] & L' \ar[r]  & M' \ar[r]  & N' \ar[r] &0 }
  \]
  It implies that $\be\xi=\xi'\al$, hence the pair $(\al,\be)$ is a morphism $\xi\to\xi'$ in the category $\cE$. On the contrary, given $(\al,\be)$
  such that $\be\xi=\xi'\al$, we obtain a commutative diagram
   \[
   \xymatrix{0\ar[r] & L \ar[r]\ar[d]^\be  & M \ar[r]\ar[d]^{\be'}  & N \ar[r]\ar@{=}[d] &0 \\
   0\ar[r] & L' \ar[r]\ar@{=}[d]  & \tilde{M}  \ar[r]\ar[d]^{\al'} & N \ar[r]\ar[d]^\al & 0\\
    0\ar[r] & L' \ar[r]  & M' \ar[r]  & N' \ar[r] &0 }
  \]
  hence a homomorphism $f=\al'\be':M\to M'$. Obviously, $f=0$ \iff $\al=\be=0$.
 \end{proof}
 
 In the next section we calculate this bimodule and present the result as a ``matrix problem''. A part of this problems turns to be a sort
 of \emph{representations of bunch of chains} \cite{bo,db}. So we can describe all indecomposable objects in this case. Then we
 calculate generators and relations for the corresponding modules, thus obtaining matrix factorizations.
  
 \section{Calculation of $\,\Ext^1$}
 \label{s2} 
 
 Recall that, if $L=\bop_iL_i$ and $N=\bop_jN_j$, an element $\xi\in\Ext^1_\qR(N,L)$ can be considered as a matrix $(\xi_{ij})$,
 where $\xi_{ij}\in\Ext^1_\qR(N_j,L_i)$. Namely, we just set $\xi_{ij}=\pi_i\xi\iota_j$, where $\pi_i$ is the projection of $L$ onto $L_i$
 and $\iota_j$ is the embedding $N_j\to N$. If we present in analogous matrix form the endomorphisms of $N$ and $L$, their action
 on $\Ext^1_\qR(N,L)$ corresponds to the usual matrix multiplication.
 
 Let $N\in\CM(\qR_{12})$ and $L\in\CM(\qR_{34})$. We can decompose $N$ and $L$ as
 \begin{equation}\label{NL} 
 \begin{split}
 M&\simeq n_1\qR_1\+ n_2\qR_2\+ n_{12}\qR_{12},\\
  L&\simeq m_3\qR_3\+ m_4\qR_4\+ m_{34}\qR_{34}.
 \end{split}
 \end{equation}
 Then an element $\xi\in\Ext^1_\qR(N,L)$ is presented by a block matrix
 \begin{equation}\label{e21} 
   \rX=\mtr{  \rX^1_3 & \rX^2_3 & \rX^{12}_3 \\
  				\rX^1_4 & \rX^2_4 & \rX^{12}_4 \\	
  				\rX^1_{34} & \rX^2_{34} & \rX^{12}_{34} },
 \end{equation}
 where $\rX^s_r$ is an $m_s\xx n_r$ matrix with elements from $\Ext^1_\qR(\qR_s,\qR_r)$ (here and later $r\in\{3,4,34\},\,s\in\{1,2,12\}$).
 So we must calculate $\Ext^1_\qR(\qR_s,\qR_r)$. Let $\qI_s$ be the kernel of the projection $\qR\to\qR_s$, so we have the exact
 sequence $0\to \qI_s\to \qR\to \qR_s\to0$, whence
  \[
  \Ext^1_\qR(\qR_s,L)\simeq \Hom_\qR(\qI_s,L)/L,
 \]
 where $L$ is emedded to $\Hom_\qR(\qI_s,L)$ if we identify an element $u\in L$ with the homomorphism $a\mapsto au$.
  
 One easily sees that $\qI_1=y\qR,\ \qI_2=x\qR$ and $\qI_{12}=xy\qR$. Hence, if we consider $L$ as embedded into 
 $\qK_{34}\*_\qR L$, we obtain the identifications
 \begin{equation}\label{ext} 
 \begin{split}
   \Ext^1_\qR(\qR_1,L)&=y^{-1}L/L\simeq L/yL,\\
 \Ext^1_\qR(\qR_2,L)&=x^{-1}L/L\simeq L/xL,\\
 \Ext^1_\qR(\qR_{12},L)&=(xy)^{-1}L/L\simeq L/xyL.
 \end{split}
 \end{equation}
 In the table nearby we present bases of the modules $\Ext^1_\qR(\qR_s,\qR_r)$. In this table $1_r$ denotes the residue class of
 the identity element of $\qR_r$ in the corresponding quotient module, $t_r$ denotes its $t$-multiple.
 \[
  \begin{array}{|c|c|c|c|}
  \hline 
 \hhh & \qR_1 & \qR_2 & \qR_{12} \\
  \hline
 \hhh \qR_3 & 1_3 & 1_3 &  1_3,\,t_3 \\
 \hline
\hhh  \qR_4 & 1_4 & 1_4 & 1_4,\,t_4 \\
  \hline
\hhh   \qR_{34}  & 1_{34}, & 1_{34},  & 1_{34},t_3,\,t_4, \\
  & t_3=-t_4  & t_3=-\la t_4&  t^2_3=-\la t^2_4 \\
  \hline
  \end{array}
 \] 
 
 The formulae \eqref{ext} imply that if $\al:L\to L'$ and $\im\al\sbe xyL'$, then $\al\xi=0$ for every $\xi\in\Ext^1_\qR(N,L)$. In the same way, if 
 $\be:N'\to N$ and $\im\be\sbe xyN$, then $\xi\be=0$ for every $\be\in\Ext^1_\qR(N,L)$. Therefore, if we are
 interesting in classification of elements of the bimodule $\Ext^1_\qR$ up to isomorphism, we can replace 
 $\Hom_\qR(\qR_a,\qR_b)$ by $\qE_{ab}=\Hom_\qR(\qR_a,\qR_b)/\Hom_\qR(\qR_a,xy\qR_b)$, where
 $a,b\in\{1,2,3,4,12,34\}$.
 An easy calculation gives the following values of $\qE_{ab}$:
  \begin{align*}
& \qE_{12,12}\simeq\qE_{34,34}\simeq \qE',\\
 & \qE_{r,r}\simeq \qE_{s,s}\simeq \qE,\\
 & \qE_{12,s}\simeq\qE_{34,r}\simeq \qE,\\
 & \qE_{s,12}\simeq \qE_{r,34} \simeq t\qE,
  \end{align*}  
 where $r\in\{3,4\},\,s\in\{1,2\}$, $\qE=\qS/(t^2),\,\qE'=\aK[[t_3,t_4]]/(t_3t_4)$. Obviously, all other values of $\qE_{ab}$ are zero. 
 
 Therefore, two matrices $\rX$ and $\rX'$ of the form \eqref{e21} describe isomorphic 
 modules \iff $S\rX=\rX' T$, where $S=(S^j_i),\ i,j\in\{3,4,34\}$ and $T=(T^j_i),\ i,j\in\{1,2,12\}$ are invertible $3\xx3$ block matrices 
 such that $S^3_3,S^4_4,S^{34}_3,S^{34}_4$, as well as $T^1_1,T^2_2,T_1^{12},T_2^{12}$ are with elements from $\qE$, 
 $S^3_4,S^4_3,T^1_2,T^2_1$ are zero, $S^{34}_{34}$ and $T^{12}_{12}$ are with elements from $\qE'$ and  
 $S^3_{34},S^4_{34}$, as well as $T^1_{12},T^2_{12}$ are with elements from $t\qE$. Symbolically:
 \[
  S,T\in\mtr{\qE&0&\qE\\ 0&\qE&\qE\\ t\qE&t\qE&\qE'}.
 \]

 \section{Modules of the first level}
 \label{s3}
 
 We say that an $\qR$-module $M$ defined by a matrix $\rX$ of the form \eqref{e21} is \emph{of the first level} if all matrices $\rX^s_r$
 are with entries from $\aK1_r$ (i.e. $t$ does not occur). Obviously, if $M'$ is another module of the first level defined by a matrix
 $\rX'$, then $M\simeq M'$ \iff $S\rX= \rX'T$, where $S,T$ are as above but with elements from $\aK$. It means that we can do elementary
 transformations inside each vertical or horizontal stripe and we can add rows (columns) of the third horizontal (vertical) stripe to the 
 rows of the first two stripes. This matrix problem can be considered as \emph{representations of a bunch of chains} in the sense of 
 \cite{bo} or \cite[Appendix~B]{db} (we use the formulation of the second paper). Namely, we have a pairs of chains:
 \[
 \kE=\{e_2<e_1\},\ \kF=\{f_2<f_1\}
 \]
 with the relation $\sim$: $e_1\sim e_1,\ f_1\sim f_1$. Namely, $e_1$ refers to the first and the second horizontal stripes, $e_1\sim e_1$
 means that there are no transformations between these stripes, $e_2$ refers to the third horizontal stripe and $e_2<e_1$ means that we 
 can add the rows of the third stripe to the rows of the first two stripes. In the same way $f_1$ refers to the first and the second vertical
 stripes and $f_2$ refers to the third vertical stripe. 
 
 Now we use the description of the indecomposable representations of this bunch of chains from \cite{bo,db}. In our case they correspond
 to the following words in the alphabeth $\{e_1,e_2,f_1,f_2,-,\sim\}$:
 \begin{itemize}
 \item   one cycle $\fW_0=e_1\sim e_1-f_1\sim f_1$;
 
 \item  one bispecial word $\fW_1=\underline{e_1}-\underline{f_1}$;
 
 \item  $8$ types of special words:
 	\begin{align*}
 	 \fW_2(n)&=(e_1\sim e_1-f_1\sim f_1)^n-\underline{e_1};\\
 	\fW_3(n)&=f_2-(e_1\sim e_1-f_1\sim f_1)^n-\underline{e_1};\\
 	\fW_4(n)&=(e_1\sim e_1-f_1\sim f_1)^n-e_1\sim e_1-\underline{f_1};\\
 	\fW_5(n)&=f_2-(e_1\sim e_1-f_1\sim f_1)^n-e_1\sim e_1-\underline{f_1}
 	\end{align*}
 	and the words $\fW_i^\top(n)\ (2\le i\le5)$ (\emph{transposed} to $\fW_i(n)$) obtained from the words $\fW_i(n)$ by replacing $e$ by $f$ 
 	and vice versa. The special ends of the words are underlined.
 	
 \item $7$ types of usual words:
 \begin{align*}
 \fW_6(n)&=(e_1\sim e_1-f_1\sim f_1)^n;\\
 \fW_7(n)&=(e_1\sim e_1-f_1\sim f_1)^n-e_2;\\
  \fW_8(n)&=(e_1\sim e_1-f_1\sim f_1)^n-e_1\sim e_1-f_2;\\
 \fW_9(n)&=f_2-(e_1\sim e_1-f_1\sim f_1)^n-e_2.
 \end{align*}
 and the transposed words $\fW_i^\top(n)\ (6\le i\le9)$ obtained from the words $\fW_i(n)$ by replacing $e$ by $f$ and vice versa. 
 \end{itemize}
 Here $\fW^n$ means $\underbrace{\fW-\fW-\dots-\fW}_{m\text{ times}}$. In all cases, except $\fW_6(n)$ and $\fW_9(n)$, 
 $n\in\mN\cup\{0\}$, while $n\in\mN$ in $\fW_6(n)$ and $\fW_9(n)$. Note that the words transposed to 
 $\fW_0,\,\fW_1,\,\fW_6(n),\,\fW_9(n),\,\fW_{10}$ coincide with their inverse words. Therefore, they give isomorphic representations, 
 so shall not be considered.
 
 \medskip
 Following the construction of indecomposable representations from \cite{bo}, we construct the matrices corresponding
 to these words. Recall that for a specail word we must also add a mark $\de\in\{+,-\}$, for a bispecial word we must add two marks 
 $\de_1,\de_2\in\{+,-\}$ and a size $n\in\mN$, for a cycle we must add a size $n$ and an eigenvalue  $\mu\notin\{0,1\}$.
 So we obtain the matrices presented in Table~1 below.
 {\scriptsize\begin{table}
 \caption{}\vspace*{-1em}
 \begin{align*} 
 \rX_0(n,\mu)&=\left(\begin{array}{c|c} I_n&J_n(\mu) \\  \hline \hhh I_n&I_n \end{array}\right)\quad (\mu\in\aK\setminus\{0,1\}); \notag\\
 \rX_1(n)^{++}&=\left(\begin{array}{c|c} I_m&J_m(0) \\  \hline \hhh I_m&I_m \end{array}\right)\quad \text{if } n=2m; \notag\\
 \rX_1(n)^{++}&=\left(\begin{array}{cc|c} I_m&0&J_m(0)\\ 0 &1&\fE_m \\  \hline \hhh I_m&0&I_m \end{array}\right)\quad \text{if } n=2m+1;
 							 \notag\\
 \rX_2(n)^+&=\left(\begin{array}{c|c} I_n&J_n(1) \\ 0&\fE_n \\  \hline \hhh I_n&I_n \end{array}\right); \notag\\
 \rX_3(n)^+&=\left(\begin{array}{cc|cc|c} I_n&0&I_n&0&0 \\ 0&1&0&1&0 \\ 0&0&0&0&1\\  
 					\hline \hhh J_n(1)&\fE_1^\top&I_n&0&0\\ \fE_n&0&0&0&1 \end{array}\right); \notag\\
 \rX_4(n)^+&=\left(\begin{array}{cc|c} I_n&0&J_n(1) \\ 0&1&\fE_n \\  \hline \hhh I_n&0&I_n \\ 0&1&0 \end{array}\right); \notag\\
 \rX_5(n)^+&=\left(\begin{array}{cc|c|c} I_n&\fE_1^\top&J_n(1)& 0 \\ 0&0&\fE_n&1 \\  \hline \hhh I_n&0&I_n&0\\
 					0&0&0&1  \end{array}\right);\\.
 \rX_6(n)&=\left(\begin{array}{c|c} I_n&J_n(1) \\  \hline \hhh I_n&I_n \end{array}\right); \notag\\
 \rX_7(n)&=\left(\begin{array}{ccc|ccc} I_m&0&0&J_m(1)&0&0\\ 0&I_k&0&0&I_k&0\\ 0&(1-\eps)\fE_k&0&\fE_m&0&0\\ \hline \hhh
 			  I_m&0&0 &I_m&0&0\\ 0&J_k(1)& \fE_1^\top&0&I_k& 0 \\ 0&\fE_k&0& \eps\fE_m&0&0 \\  \hline \hhh 
 				0&0&1&0&0&1  \end{array}\right), \notag\\
 		          &\hspace*{2em} \text{where } m=[n/2],\ k=n-m-1,\\
 				 &\hspace*{2em}\eps=1 \text{ if $n$ is odd and $\eps=0$ if $n$ is even};\\
  \rX_8(n)&=\left(\begin{array}{c|c|c} I_n&J_n(1)&0 \\ 0&\fE_n&1 \\  \hline \hhh I_n&I_n&0 \\ 0&0&1 \end{array}\right); \notag\\
 \rX_9(n)&=\left(\begin{array}{cc|cc|c} I_n&0&I_n&0&0\\ 0&0&0&0&1\\ \hline \hhh J_n(1)&\fE_1^\top&I_n&0&0 \\ \fE_n&0&0&0&1 \\  
 \hline \hhh 0&1&0&1&0  \end{array}\right). \notag\\
 \rX_{10}&=\mtr{1}. \notag
 \end{align*}
 \end{table}}
 In this table $I_n$ is the unit $n\xx n$ matrix, $J_n(\mu)$ is the lower $n\xx n$ Jordan matrix with the eigenvalue $\mu$, $\fE_n=(0,0,\dots,0,1)$,
 $\fE_1=(1,0,\dots,0)$ and $^\top$ means the transposition. If $n=0$, $m=0$ or $k=0$, the corresponding rows and columns are absent. 
 All matrices are subdivided just as the matrix $\rX$ in \eqref{e21}. If there are only two vertical (horizontal) stripes, they correspond
 to the first two vertical (horizontal) stripes of $\rX$ (no $\rX^{12}_r$ or $\rX^s_{34}$). The only exception is $\rX_{10}$, which consists 
 of the unique part $\rX^{12}_{34}$. If we change the mark $+$ to $-$, we interchange the first two stripes (horizontal if the special end is 
 $e_1$ and vertical if it is $f_1$). Then we call the new matrices the \emph{$\pm$-modidfications}. The matrices corresponding to the 
 transposed words are just the transposed matrices. 
 
 We denote by $M_k(n)$, adding $+$ or $-$ if necessary, the module corresponding to the matrix $\rX_k(n)$ from Table~1 and their
 $\pm$-modifications. By $M^*_k(n)$ we denote the module corresponding to $\rX_k(n)^\top$. Thus we have the following result. 
 
 \begin{theorem}\label{main} 
 Any $\qR$-module of the first level decomposes into a direct sum of modules $M_k(n)$ and $M^*_k(n)$.
 \end{theorem}
 
 \section{Generators and relations}
 \label{s4} 
   
  Now we calculate matrix factorizations of the polynomial $F=XY(X-Y)(X-\la Y)$ corresponding to the indecomposable modules of the
  first level. In other words, we find minimal sets of generators for these modules and minimal sets of relations for these generators.
 We present detailed calculations for the case of the modules corresponding to the matrix $X^+_2(2)$ and its transposed. 
 It is quite typical though not very cumbersome. In order to make smaller the arising matrices, we denote by $z_i\ (1\le i\le 4)$ 
 the generators of the kernel of the projection  $\qR\to\qR_i$, that is $z_1=y,\,z_2=x,\,z_3=x-y,\,z_4=x-\la y$. Thus $F=z_1z_2z_3z_4$.
 
 Note that, if we have a decomposition like \eqref{NL} of the modules $N$ and $L$ and choose generators $v^s_j$ of the summands 
 $\qR^s$ of $N$ and $u^r_i$ of the summands $\qR^r$ of $L$, then $\{v^s_j,u^r_i\}$ is a set of generators of $M$. In this case the relation for
 $u^r_i$ is $z_ru_i=0$. The relations for $v_i^s$ are obtained from the $j$-th columns of the matrices $\rX^s_r$. Namely, if the 
 coefficients  of these columns are $\xi^r_i\in\Ext^1_\qR(\qR_s,\qR_r)$, where $\xi^r_i$ are identified with elements from 
 $\qR_r/z_s\qR_r$ as in \eqref{ext}, then the relation for $v^s_j$ is 
 \[
  z_sv^s_j=\sum_{r,i} \xi_i^r u^r_i.
 \]
 
 As the matrix $\rX^+_2(2)$ equals 
 {\scriptsize
 \[
 \left( \begin{array}{cc|cc}
  1&0 & 1&0 \\ 0&1 & 1&1 \\ 0&0 &0&1 \\
  \hline
  \hhh 
 1&0 & 1&0 \\ 0&1 & 0&1
  \end{array} \right)
 \]}%
 generators for the module $M_2^+(n)$ are $v^1_1,v^1_2,v^2_1,v^2_2;u^3_1,u^3_2,u^3_3,u^4_1,u^4_2$ with the relations
 {\small
\begin{align*}
 z_3u^3_i&=0, & z_4u^4_i&=0,\\
 z_1v^1_1&=u^3_1+u^4_1, & z_2v^2_1&=u^3_1+u^3_2+u^4_1=z_1v^1_1+u^3_2,\\
 z_1v^1_2&=u^3_2+u^4_2, & z_2v^2_2&=u^3_2+u^3_3+u^4_2=z_1v^1_2+u^3_3.
\end{align*}}%
One easily sees that the elements $u^3_2,u^3_3,u^4_1,u^4_2$ can be excluded from the list of relations. Afterwords, the relations
 become
 {\small
 \begin{align*}
 &z_3u^3_1=0, && \\
 &z_4(z_1v^1_1-u^3_1)=0, & z_3(z_2v^2_1-z_1v^1_1)&=0,\\
 &z_4(z_1v^1_2-z_2v^2_1+z_1v^1_1)=0, & z_3(z_2v^2_2-z_1v^1_2)&=0.
 \end{align*}}%
 Therefore, if we order the remaining generators as $u^3_1,v^1_1,v^2_1,v^1_2,v^2_2$, the matrix of relations becomes
 {\scriptsize
 \[
  \Phi=\mtr{z_3 &0&0&0&0 \\ -z_4 & z_1z_4 & 0&0&0 \\ 0 & -z_1z_3 & z_2z_3 & 0&0 \\
  		0& z_1z_4 & -z_2z_4 & z_1z_4 & 0 \\ 0&0&0& -z_1z_3 & z_2z_3 }.
 \]}%
 As all coefficients are from the radical, this presentation is minimal.
 The matrix $\Psi$ giving the matrix factorization $\Psi\Phi=FI_5$ can be easily calculated:
 {\scriptsize
 \[
  \Psi=\mtr{z_1z_2z_4 &0&0&0&0\\ z_2z_4 &z_2z_3 &0&0&0 \\ z_1z_4 & z_1z_3 & z_1z_4 & 0&0\\
  				0&0& z_2z_4 & z_2z_3 & 0 \\ 0&0& z_1z_4 & z_1z_3 & z_1z_4 }.
 \]}
 
 The matrix $\rX^+_2(2)^\top$ equals
 {\scriptsize
 \[
  \left(\begin{array}{ccc|cc}
  1&0&0& 1&0\\ 0&1&0 & 0&1\\ \hline \hhh 1&1&0 & 1&0\\ 0&1&1 &1&1
  \end{array}\right)
 \]}%
 Here all generators $u^r_i$ can be expressed through $v^s_j$. Arranging the latter as $v^1_3,v^2_2,v^1_2,v^2_1,v^1_1$,
 we obtain the  matrix of relations for the module $M^*_2(n)^+$:
 {\scriptsize
 \[
  \Phi^*=\mtr{z_1z_4 &0&0&0&0 \\ -z_1z_3 & z_2z_3 & 0&0&0\\ 0& -z_2z_4 & z_1z_4 & 0&0\\
  					0& z_2z_3 & -z_1z_3 & z_2z_3 & 0 \\ 0&0&0& -z_2 & z_1 }
 \]}%
 The matrix $\Psi^*$ such that $\Psi^*\Phi^*=FI_5$ equals
 {\scriptsize
 \[
  \Psi^*=\mtr{z_2z_3 &0&0&0&0 \\ z_1z_3 & z_1z_4 & 0&0&0\\ z_2z_3& z_2z_4 & z_2z_3 & 0&0\\
  					0& 0 & z_1z_3 & z_1z_4 & 0 \\ 0&0&z_2z_3 & z_2z_4 & z_2z_3z_4 }
 \]}%
 The same calculations can be done for the modules corresponding to the matrices $X^+_2(n)$ and $X^+_2(n)^\top\!$, which gives the
 matrices of relations, respectively,
 {\scriptsize
 \begin{align*}
 &  \Phi_2(n)^+=\mtr{ z_3 & 0&0&0&0&0& \dots & 0&0 \\
   				-z_4 & z_1z_4& 0&0&0&0&\dots&0&0\\
   				0& -z_1z_3 & z_2z_3 & 0&0&0&\dots&0&0\\
   				0& z_1z_4 &-z_2z_4& z_1z_4&0 &0&\dots&0&0\\
   				0&0&0& -z_1z_3 & z_2z_3&0 &\dots& 0&0\\
   				0&0&0&z_1z_4& -z_2z_4 &z_1z_4&\dots&0&0\\
   				\hdotsfor{9}\\
   				0&0&0&0&0&0&\dots & -z_1z_3& z_2z_3}\\
\intertext{\normalsize and}
 & \Phi^*_2(n)^+=\mtr{z_1z_4 &0&0&0&\dots &0&0&0&0\\
 				-z_1z_3 &z_2z_3 &0&0& \dots & 0&0&0&0 \\   0&-z_2z_4 & z_1z_4 &0 &\dots &0&0&0&0 \\
 				0&z_2z_3 & -z_1z_3 & z_2z_3 &\dots& 0&0&0&0 \\ \hdotsfor9 \\
 				0&0&0&0& \dots &  -z_2z_4 & z_1z_4 &0 &0\\ 0&0&0&0&\dots& z_2z_3 & -z_1z_3 & z_2z_3 & 0\\
 				0&0&0&0& \dots & 0&0& -z_2& z_1}.
 \end{align*}}%
 of size $(2n+1)\xx(2n+1)$.
 
  Quite analogously one can calculate matrices of relations $\Phi_k(n)$ and $\Phi^*_k(n)$ for the modules,respectively, $M_k(n)$ and 
  $M_k^*(n)$. The results are given in Table~2 for $0\le k\le 9$. The module $M_{10}$ is actually the regular module and its 
  matrix of relations is just $(F)$. Most of these matrices are modifications of $\Phi_2(n)^+$ and $\Phi_2^*(n)^+$, so we only mention 
  the necessary changes. For instance, the matrix $\Phi_3(n)^+$ described in Table~2 is indeed
  {\scriptsize
  \[
   \Phi_3(n)^+=\mtr{	z_2z_3 &0&0& \dots & 0&0&0 \\   -z_2z_4 & z_1z_4 &0 &\dots &0&0&0 \\
 				z_2z_3 & -z_1z_3 & z_2z_3 &\dots& 0&0&0 \\ \hdotsfor7 \\
 				0&0&0& \dots &  -z_2z_4 & z_1z_4 &0 \\ 0&0&0&\dots& z_2z_3 & -z_1z_3 & z_1z_2z_3 }
  \]}%
  of size $(2n+3)\xx(2n+3)$.
    If we change the mark $+$ to $-$, one must interchange the multipliers $z_1$ and $z_2$ if the special end was $e_1$ and interchange the 
 multipliers $z_3$ and $z_4$ if the special end was $f_1$. 

\def\fff{\phantom{^+}}
    \begin{table}[!ht]
 \caption{}\vspace*{-1em}
 {\scriptsize
 \begin{align*}
 \Phi_0(n,\mu)&=\mtr{\phi(\mu) & 0 & 0 &\dots & 0& 0\\ \phi'(\mu) & \phi(\mu) & 0 &\dots &0 &0 \\ 
 				0& \phi'(\mu) & \phi(\mu) &\dots &0&0\\ \hdotsfor6\\
 				0&0&0&\dots &\phi(\mu) & 0\\ 0&0&0&\dots& \phi'(\mu)&\phi(\mu) },\\
 				&\text{where }\ \phi(\mu)=\mtr{z_1z_3 & -z_2z_3 \\ -\mu z_1z_4 & z_2z_4 }\!,\
 				\phi'(\mu)=(\mu-1)^{-1}\mtr{z_1z_3&-z_2z_3 \\ -z_1z_3&z_2z_3};\\
 \Phi_1(n)^{++}&=\mtr{z_2z_4 & 0&0&\dots &0&0 \\
 						-z_2z_3 & z_1z_3 & 0&\dots&0&0\\
 						z_2z_4& -z_1z_4 &z_2z_4 &\dots	&0&0\\
 						\hdotsfor6\\
 						z_2z_4&-z_1z_4 &z_2z_4&\dots& z_2z_4 &0\\
 						-z_2z_3 & z_1z_3 &-z_2z_3 &\dots & -z_2z_3& z_1z_3}\\
  &\ \text{if $n$ is even; if $n$ is odd, just cross out the first row and the first column};	\\
  \Phi_3(n)^+&\text{ is obtained from $\Phi^*_2(n+2)^+$ by deleting the first and the last rows and columns}\\
  &\text{ and then multiplying the last column by $z_1$};\\
  \Phi_3^*(n)^+&\text{ is obtained from $\Phi_2(n+2)^+$ by deleting the first and the last rows and columns}\\
  &\text{ and then multiplying the first row by $z_3$};\\
  \Phi_4(n)^+&\text{ is obtained from $\Phi_2(n+1)^+$ by deleting the last row and the last column};	\\
  \Phi^*_4(n)^+&\text{ is obtained from $\Phi^*_2(n+1)^+$ by deleting the first row and the first column};	\\
  \Phi_5(n)^+&\text{ is obtained from $\Phi_4(n)^+$ by multiplying the first column by $z_1$}\\ &\text{ and the last column by $z_2$};\\
  \Phi_5^*(n)^+&\text{ is obtained from $\Phi^*_4(n)^+$ by multiplying the first and the last rows by $z_4$};\\
  \Phi_6(n)\fff &\text{ is obtained from $\Phi_2(n)^+$ by dividing the last row by $z_3$}\\
  \Phi_7(n)\fff &=
  \begin{cases}
   \mtr{\Phi_2(m)^+ & 0 \\ \fU_7 & \Phi'_7(k)} & \text{ if $n$ is even},\\
   \mtr{\Phi'_7(k)&0 \\ \fU'_7 &\Phi_2(m)^+} & \text{ if $n$ is odd},
  \end{cases} \\
  &\text{ where $\fU_7$ and $\fU_7'$ are zero except the last row which is}\\    
  & \text{ $z_2z_4\fE_{2m+1}$ for $\fU_7$ and $z_1z_3\fE_{2k+1}$ for $\fU'_7$, while}\\ 
  &\text{ $\Phi'_7(k)$ is obtained from $\Phi^*_4(k)^+$ by multiplying the first and the last rows by $z_4$};\\
   \Phi^*_7(n)\fff &=
  \begin{cases}
   \mtr{\Phi_2^*(m)^+ & 0 \\ \fU^*_7 & {\Phi^*_7}'(k)} & \text{if $n$ is odd},\\
   \mtr{{\Phi^*_7}'(k)&0 \\ {\fU^*_7}' &\Phi_2^*(m)^+} & \text{if $n$ is even},
  \end{cases} \\
  &\text{ where $\fU^*_7$ and ${\fU^*_7}'$ are zero except the first column which is}\\    
  & \text{ $-z_1z_3\fE_1^\top$ for $\fU^*_7$ and $-z_2z_4\fE_1^\top$ for ${\fU^*_7}$, while}\\ 
  &\text{ ${\Phi^*_7}'(k)$ is obtained from $\Phi_4(k)^+$ by multiplying the first and the last column by }z_2 ;\\
  \Phi_8(n)\fff &\text{ is obtained from $\Phi_4(n)^+$ by multiplying the last column by $z_2$};\\
  \Phi^*_8(n)\fff &\text{ is obtained from $\Phi_4^*(n)^+$ by multiplying the first row by $z_4$};\\
 \Phi_9(n)\fff &\text{ is obtained from $\Phi^*_8(n+1)$ by deleting the last column and the last row}\\
 			&\text{ and then multiplying the last column by $z_1$.}
\end{align*}}
  \end{table}
  
\section{Remarks}

Unfortunately, we cannot give an internal characterization of modules of the first level. Nevertheless, the following conjecture seems
very plausible.

\begin{conj}\label{conj} 
 If $M$ is an indecomposable module of the first level, its Auslander--Reiten translation $\tau M$ is also of the first level.
\end{conj}

 Recall that if a matrix $\Phi$ of size $d\xx d$ is the matrix of relations for a module $M$, $\Psi$ is such that $\Phi\Psi=FI_d$, then
 $\Psi$ is the matrix of relations for $\tau M$ \cite[Lemma 9.8]{yo}. In particular, $\tau^2=\mathrm{id}$. Note that $\det\Phi\det\Psi=F^d$.
If Conjecture~\ref{conj} is true, a simple comparison of determinants of matrices from Table~2 gives the following formulae for
the Auslander--Reiten translations of indecomposable modules of the first level:
\begin{align*}
 & \tau M=M \text{ if } M\in\{M_0(n,\mu),\,M_8(n),\,M^*_8(n)\},\\
 & \tau M_1(n)^{++}= M_1(n)^{--},\\
 & \tau M_1(n)^{+-}= M_1(n)^{-+},\\ 
 & \tau M_2(n)^\pm= M_3(n-1)^\mp,\\
 & \tau M_2^*(n)^\pm=M_3^*(n-1)^\mp,\\
 & \tau M_4(n)^\pm=M_5(n)^\mp,\\
 & \tau M_6(n)=M_9(n-1),\\
 & \tau M_7(n)=M_7^*(n).
\end{align*}
 For instance, the matrix $\Phi_4(n)^+$ is of size $(2n+2)\xx(2n+2)$ an $\det\Phi_n(4)^+=z_1^{n+1}z_2^{n}z_3^{n+1}z_4^{n+1}$. 
 Hence, if $\Psi$ is a matrix of relations for $\tau M_4(n)^+$, then $\det\Psi=z_1^{n+1}z_2^{n+2}z_3^{n+1}z_4^{n+1}$, 
 which only coincides with $\det\Phi_5(n)^-$.


 \end{document}